\documentclass{article}
\usepackage{tikz}
\usepackage{grfext}
\usepackage{listings}
\usepackage{color}
\usepackage{graphicx}            
\usepackage{amsthm,amsfonts,amsmath,amssymb}
\usepackage{array,blkarray,multirow, enumerate}
\usepackage{xcolor}
\usepackage{authblk}
\usepackage{caption, setspace}

\DeclareMathSymbol{\mrq}{\mathord}{operators}{`'}

\newtheorem{thm}{Theorem}
\newtheorem{lemma}{Lemma}

\newtheorem*{obs*}{Observation}
\newtheorem{prop}{Proposition}

\newtheorem{prob}{Problem}
\newtheorem{definition}{Definition}
\newtheorem{claim}{Claim}

\newcommand{\cG}{{\cal G}}
\newcommand{\cF}{{\cal F}}

\newcommand{\cS}{{\cal S}}
\newcommand{\cH}{{\cal H}}
\newcommand{\ceil}[1]{\left\lceil{#1}\right\rceil}

\newcounter{mycount}

\def\0{\mathbf{0}}

\usepackage[normalem]{ulem}

\begin{document}
\lstset{language=Python}          

\title{On the equality of domination number and $ 2 $-domination number}
\author{G\"{u}lnaz Boruzanl{\i} Ekinci\footnote{Department of Mathematics, Faculty of Science, Ege University,  Izmir, Turkey\\E-mail address: gulnaz.boruzanli@ege.edu.tr} \hspace{0.2cm}and\hspace{0.2cm}Csilla Bujt\'as\footnote{Faculty of Information Technology, University of Pannonia, Veszpr\'em, Hungary; and Faculty of Mathematics and Physics, University of Ljubljana, Slovenia\\E-mail address: bujtas@dcs.uni-pannon.hu}}

\maketitle
\begin{abstract}
The $2$-domination number $\gamma_2(G)$ of a graph $G$ is the minimum cardinality of a set $ D \subseteq V(G) $ for which
every vertex outside $ D $ is adjacent to at least two vertices in $ D $. Clearly,  $ \gamma_2(G) $ cannot be smaller than the 
domination number  $ \gamma(G) $. We consider a large class of graphs and characterize those members which satisfy 
$\gamma_2=\gamma$. For the general case, we prove that it is NP-hard to decide whether $\gamma_2=\gamma$ holds. We also 
give a necessary and sufficient condition for a graph to satisfy the equality hereditarily.

\end{abstract}

\bigskip


\bigskip


\section{Introduction}
In this paper, we continue to expand on the study of graphs that satisfy the equality $\gamma(G) = \gamma_2(G)$, where $\gamma(G)$ and $\gamma_2(G)$ stand for the domination number and the $ 2 $-domination number of a graph $ G $, respectively. If $\gamma(G) = \gamma_2(G)$ holds for a graph $G$, then we call it $ (\gamma,\gamma_2) $\textit{-graph}. We prove that the corresponding recognition problem is NP-hard and there is no forbidden subgraph characterization for $ (\gamma,\gamma_2) $-graphs in general. On the other hand, in one of our main results, we consider a large graph class $\cH$ and give a special type of forbidden subgraph characterization for $ (\gamma,\gamma_2) $-graphs over $\cH$. Although the number of these forbidden subgraphs is infinite, we prove that the recognition problem is solvable in polynomial time on $\cH$. 
Putting the question into another setting, we give a complete characterization for $(\gamma, \gamma_2)$-perfect graphs, that is, we characterize the graphs for which all induced subgraphs with minimum degree at least two satisfy the equality of domination number and $ 2 $-domination number. 

\subsection{Terminology and Notation}

\indent Let $ G $ be a simple undirected graph, where $ V(G) $ and $ E(G) $ denote the set of vertices and the set of edges of $ G $, respectively. The \textit{(open) neighborhood} of a vertex $ v $ is the set $ N_G(v) = \{u \in V(G): uv \in E(G)\} $ and its \textit{closed neighborhood} is $ N_G[v] = N_G(v) \cup\{v\}$. The \textit{degree} of $ v $ is given by the cardinality of $ N_G(v) $, that is, $ \deg_G(v)  =|N_G(v)| $. We will write $N(v)$, $N[v]$ and $\deg(v)$ instead of $N_G(v)$, $N_G[v]$ and $\deg_G(v)$, if $G$ is clear from the context. An edge $ uv $ is a \textit{pendant edge} if $ \deg(u)=1 $ or $ \deg(v)=1 $, otherwise the edge is \textit{non-pendant}. The minimum and maximum vertex degrees of $ G $ are denoted by $ \delta(G) $ and $ \Delta(G) $, respectively. For a subset $ S\subseteq V(G) $, let $ G[S] $ denote the subgraph induced by $S$. We say that $S$ is \textit{independent} if $G[S]$ does not contain any edges. For disjoint subsets $ U, W \subseteq V(G)$, we let $ E[U,W] $ denote the set of edges between $ U $ and $ W $. 

For a positive integer $ k $, the \textit{$ k^{th} $ power} of a graph $ G $, denoted by $ G^k $, is the graph on the same vertex set as $ G $ such that $ uv $ is an edge if and only if the distance between $ u $ and $ v $ is at most $ k $ in $ G $. An edge $ uv \in E(G) $ is \textit{subdivided} by deleting the edge $ uv $, then adding a new vertex $ x $ and two new edges $ ux $ and $ xv $. Let $ K_n $, $ C_n $ and $P_n$ denote the complete graph, the cycle and the path, all of order $ n $, respectively; and let $ S_n $ denote the star of order  $ n+1 $. For any positive integer $ n $, let $ [n] $ be the set of positive integers not exceeding $ n $. For notation and terminology not defined here, we refer the reader to \cite{West2001}.

For a positive integer $ k $, a subset $ D \subseteq V(G) $ is a \textit{$ k $-dominating set} of the graph $ G $ if $ |N_G(v) \cap D|\geq k $ for every $ v \in V(G)\setminus D $. The \textit{$ k $-domination number} of $ G $, denoted by $ \gamma_k(G) $,  is the minimum cardinality among the $ k $-dominating sets of $ G $. Note that the $ 1 $-domination number, $ \gamma_1(G) $, is the classical domination number $ \gamma(G) $.

A graph $G$ is called \textit{$F$-free} if it does not contain any induced subgraph isomorphic to $F$. More generally, let $\cF$ be a (finite or infinite) class of graphs, then $G$ is $\cF$-free if it is $F$-free for all $F\in \cF$. On the other hand, let $G^D$ denote a graph $G$ with a specified subset $D \subseteq V(G)$. Then, $F^{D'}$ is a (induced) subgraph of $G^D$ if $F$ is a (induced) subgraph of $G$ and $D'=V(F)\cap D$. We say that $F_1^{D_1}$ is isomorphic to $F_2^{D_2}$ if there is an edge-preserving bijection between $V(F_1)$ and $V(F_2)$ which maps $D_1$ onto $D_2$. Analogously, we may define the $F^{D'}$-freeness of $G^D$ and forbidden (induced) subgraph characterization with a specified vertex subset $D$. 

\subsection{Preliminary results}
The concept of $ k $-domination in graphs was introduced by Fink and Jacobson \cite{Fink85, Fink85-2} and it has been studied extensively by many researchers (see for example \cite{bonomo2018, Bujtas2017, Caro1990-2, Caro1990, desormeaux2014, Favaron1988, Favaron2008, Hansberg2015, Hansberg2013, krzywkowski2017, Shaheen2009, yue2020}). For more details, we refer the reader to the books on domination by Haynes, Hedetniemi and Slater \cite{ DominationBook2, DominationBook1} and to the survey on $ k $-domination and $ k $-independence by Chellali \textit{et al.}\ \cite{ChellaliSurvey}.

Fink and Jacobson \cite{Fink85} established the following basic theorem.  
\begin{thm}\label{thm:FJ} \cite{Fink85}
	For any graph $ G $ with $ \Delta(G)\geq k\geq 2 $, $ \gamma_k(G) \geq \gamma(G)+k-2$.
\end{thm}
Although it is proved that the above inequality  is sharp for every $k\ge 2$, the characterization of graphs attaining the equality is still open, even for the case when  $ k = 2$. The corresponding characterization problem was studied in~\cite{Hansberg2015, Hansberg2016, Hansberg2008}, 
while similar problems involving different domination-type graph and hypergraph invariants were considered for example in~\cite{Arumugam2013, Blidia2006, hartnell1995, krzywkowski2017, Randerath1998}.

In this paper, we study $ (\gamma,\gamma_2) $-graphs that is graphs for which Theorem~\ref{thm:FJ} holds with equality if $k=2$. Note that $ G $ is a $ (\gamma,\gamma_2) $-graph, that is $\gamma_2(G)=\gamma(G)$, if and only if every component  of $ G $ is a $ (\gamma,\gamma_2) $-graph. Thus, we only deal with connected graphs in the rest of the paper.

Hansberg and Volkmann \cite{Hansberg2008} characterized the cactus graphs (i.e., graphs in which no two cycles share an edge)  which are $ (\gamma,\gamma_2) $-graphs and they also gave some general properties of the graphs attaining the equality. In 2016, the claw-free (i.e., $S_3$-free) $ (\gamma,\gamma_2) $-graphs and the line graphs which are $ (\gamma,\gamma_2) $-graphs were characterized  by Hansberg \textit{et al.} \cite{Hansberg2016}. We will refer to the following basic lemmas proved in these papers.

\begin{lemma} \cite{Hansberg2008} \label{lem:0}
	If $ G $ is a connected nontrivial graph with $ \gamma_2(G)=\gamma(G) $, then $ \delta(G)\geq 2 $.
\end{lemma}

\begin{lemma}\cite{Hansberg2016} \label{lem:1}
	Let $D$ be a minimum $ 2 $-dominating set of a graph $G$. If $ \gamma_2(G)=\gamma(G) $, then $ D $ is independent.
\end{lemma}

\begin{lemma}\cite{Hansberg2016}  \label{lem:2}
	Let $ G $ be a connected nontrivial graph with $ \gamma_2(G)=\gamma(G) $ and let $ D $ be a minimum $ 2 $-dominating set of $ G $. Then, for each vertex $ u' \in V\setminus D $ and $ u,v \in D \cap N(u') $, there is a vertex $ v' \in V \setminus D $ such that $ u,u',v $ and $ v'  $ induce a $ C_4 $.
\end{lemma}

We strengthen Lemma \ref{lem:2} by proving the following statement.

\begin{lemma}
	\label{lem:2nghbrs}
	Let $ G $ be a connected nontrivial graph with $ \gamma_2(G)=\gamma(G) $ and let $ D $ be a minimum $ 2 $-dominating set of $ G $. For every pair $ u,v \in D$, if $ N_G(u) \cap N_G(v) \neq \emptyset $, then there exists a nonadjacent pair $ u',v' \in V\setminus D $ such that $ N_G(u') \cap D =  N_G(v')\cap D = \{u,v\}$.
\end{lemma}
\begin{proof}
	For every vertex $ x \in N_G(u)\cap N_G(v)  $, there is a vertex $ y  $ different from $x$ such that $N_G(y)\cap D = \{u,v\}$ and $xy \notin E(G)$, since otherwise  $ (D \setminus \{u,v\})\cup\{x \} $ would be a dominating set of $ G $, a contradiction. This proves that we have at least two non-adjacent vertices $u'$ and $v'$ with the property $ N_G(u') \cap D =  N_G(v')\cap D = \{u,v\}$.
\end{proof}

The following simple proposition demonstrates that $(\gamma,\gamma_2)$-graphs form a rich class and it indicates the possible difficulties in a general characterization.

\begin{prop}
	\label{prop:0}
	There is no forbidden (induced) subgraph for the graphs satisfying the equality of domination number and $ 2 $-domination number.
\end{prop}  
\begin{proof} Consider an arbitrary graph $F$ and a four-cycle $ C_4 $, which is vertex-disjoint to $F$. Let $ u $ and $ v $ be two non-adjacent vertices of $ C_4$. Construct the graph $ G_F $ by joining each vertex of $ F $ to both $ u $ and $ v $. Since, for any $F$, the graph $ G_F$ contains $ F $ as an induced subgraph and it satisfies the equality $ \gamma_2(G_F) = \gamma(G_F) =2$, there is no forbidden induced  subgraph for $(\gamma,\gamma_2)$-graphs.
\end{proof}

As a consequence of the Lemmas \ref{lem:0}-\ref{lem:2nghbrs}, we will prove that all $(\gamma,\gamma_2)$-graphs belong to the following graph class $\cG$ that we define together with its subclasses $\cG_1$ and $\cG_2$.
\begin{definition} Given an arbitrary simple graph $F$ with vertex set $V(F)=D=\{v_1,\dots v_d\}$, a graph $G$ belongs to the class $\cG(F)$ if $G$ can be obtained from $F$ by the following rules.
	\begin{itemize}
	
	\item[$(i)$] Define a pair of vertices $X_{i,j}=\{x_{i,j}^1, x_{i,j}^2\}$ for every edge $v_iv_j$ of $F$, and further, let $Y$ be an arbitrary (possibly empty) set of vertices, such that $D$, $Y$ and all the pairs $X_{i,j}$ are mutually disjoint sets of vertices. Define  $V(G)=D \cup X \cup Y$, where $X=\bigcup_{v_iv_j\in E(F)}X_{i,j}$. 
	\item[$(ii)$] The edges between $D$ and $X\cup Y$ are defined such that  $N_G(x_{i,j}^s)\cap D=\{v_i,v_j\}$ for every vertex $x_{i,j}^s\in X$, and the set $N_G(u)\cap D$ contains at least two vertices and induces a complete subgraph in $F$ for any $u\in Y$. The induced subgraph $G[D]$ cannot contain edges.
	\item[$(iii)$] The edges inside $X\cup Y$ can be chosen arbitrarily, but each $X_{i,j}$ must remain independent. 
\end{itemize}
	Moreover, $G$ belongs to $\cG_1(F)$ if $ |N_G(y)\cap D| = 2 $ for each $ y \in Y $; and $G$ belongs to $\cG_2(F)$ if $Y=\emptyset$. The graph classes $\cG$, $\cG_1$, $\cG_2$ contain those graphs $G$ for which there exists a graph $F$ such that $G$ belongs to $\cG(F)$, $\cG_1(F)$, $\cG_2(F)$, respectively. 
\end{definition}

For $G\in \cG(F)$ with the fixed partition $V(G)=D\cup X\cup Y$ as per above definition, a vertex $v$ is a \textit{$D$-vertex} (or original vertex) if $v\in D$; $v$ is a \textit{subdivision vertex} if $v\in X$; and $v$ is a \textit{supplementary vertex} if $v\in Y$. The edges inside $G[X\cup Y]$ are called \textit{supplementary edges}, and $F$ is said to be the \textit{underlying graph} of $G$. In Section 5, we will show that the underlying graph is not necessarily unique by presenting a $(\gamma, \gamma_2)$-graph having two non-isomorphic underlying graphs. Note that the construction in the proof of Proposition \ref{prop:0} always belongs to the class $ \cG_1 $. Hence, Proposition \ref{prop:0} remains true under the condition $ G \in \cG_1 $. This motivates us to focus on the smaller class $\cG_2$.

Alternatively, we may define the graph class $\cG_2(F)$ in the following constructive way. Let $ F $ be a simple graph with vertex set $ V(F) $ and edge set $ E(F) $. Consider the \textit{double subdivision graph} $ F^* $ obtained by substituting each edge $ v_iv_j $ by two parallel edges and subdividing each edge once by adding the vertices $ x_{i,j}^1 $ and $ x_{i,j}^2 $. Let $ X_{i,j}  = \{x_{i,j}^1, x_{i,j}^2\} $ and define the set of subdivision vertices $X = \bigcup_{v_iv_j \in E(F)}^{} X_{i,j}$. The graph class $\cG_2(F)$ consists of the graphs obtained by adding some (maybe zero) supplementary edges between subdivision vertices of $ F^* $ such that each $ X_{i,j}$ remains independent. 

\begin{prop}
	\label{prop:1}
	If $G$ is a graph with $ \gamma_2(G) = \gamma(G) $, then  $G \in \mathcal{G} $.
\end{prop}
\begin{proof} Assuming $ \gamma_2(G) = \gamma(G) $, choose a minimum $ 2 $-dominating set $D$ of $G$ and define the graph $ F = G^2[D] $. We first note that, by Lemma~\ref{lem:1}, $D$ is independent in $G$. Since $D$ is a $ 2 $-dominating set, every $u\in V(G)\setminus D$ has at least two neighbors in $D$ and, by the definition of $F$, the set $N_G(u)\cap D$ induces a complete subgraph in $F$. By Lemma~\ref{lem:2nghbrs}, for every edge $v_iv_j$ of $F$, there exist at least two different and non-adjacent vertices $u$, $u' \in V(G) \setminus D$ such that $ N_G(u) \cap D =  N_G(u') \cap D = \{v_i,v_j\}$. If we select such a pair and define  $X_{i,j}=\{u, u'\}$ for every $v_iv_j \in E(F)$, and let $Y=V(G)\setminus (D\cup X)$, then $G$ can be obtained from the underlying graph $F$ with the vertex partition $V(G)=D\cup X\cup Y$, proving that $G\in \cG(F)$.
\end{proof}

In a follow-up paper of the present work \cite{ekinci2020}, we studied the analogous problem for each $k \ge 3$. There we gave a characterization for connected bipartite graphs satisfying $\gamma_k(G)=\gamma(G)+k-2$ and $\Delta(G) \ge k$. This result is based on the notion of the $k$-uniform ``underlying hypergraph'' that corresponds to the underlying graph, as defined here, if $k=2$.

\subsection{Structure of the paper}

In Section 2, we define the class $\cH$ of those graphs which are contained in $\cG_2$ with an underlying graph of girth at least 5 and we give a characterization for $(\gamma, \gamma_2)$-graphs over $\cH$. Then, in Section 3, we discuss algorithmic complexity questions. First, we prove that the recognition problem of $(\gamma, \gamma_2)$-graphs is NP-hard on $\cG_1$ (even if a minimum $ 2 $-dominating set is given together with the problem instance). Then, on the positive side, we show that there is a polynomial-time algorithm which recognizes $(\gamma, \gamma_2)$-graphs over the class $\cH$ if the instance is given together with the minimum $ 2 $-dominating set $D=V(F)$. The algorithm is based on our characterization theorem and Edmond's Blossom Algorithm. In Section 4, we consider the hereditary version of the property and characterize $(\gamma, \gamma_2)$-perfect graphs. As a direct consequence, we get that $(\gamma, \gamma_2)$-perfect graphs are easy to recognize. In the concluding section, we put remarks on the underlying graphs and discuss some open problems.

\section{Characterization of $(\gamma, \gamma_2)$-graphs over $\cH$} \label{sec:2}

To formulate the main result of this section,  we will refer to the following definitions.
\begin{definition} Let $\cH$ be the union of those graph classes $\cG_2(F)$ where  the underlying graph $F$ is $(C_3,C_4)$-free.
\end{definition}

When we consider a graph $G\in \cH$, we will always assume that a fixed $(C_3,C_4)$-free underlying graph $F$ and a corresponding partition $V(G)=D\cup X$ are given. In order to indicate this structure, we will use the notation $ G^D $.

\begin{figure}[h]
	\centerline{\includegraphics[width=0.4\textwidth]{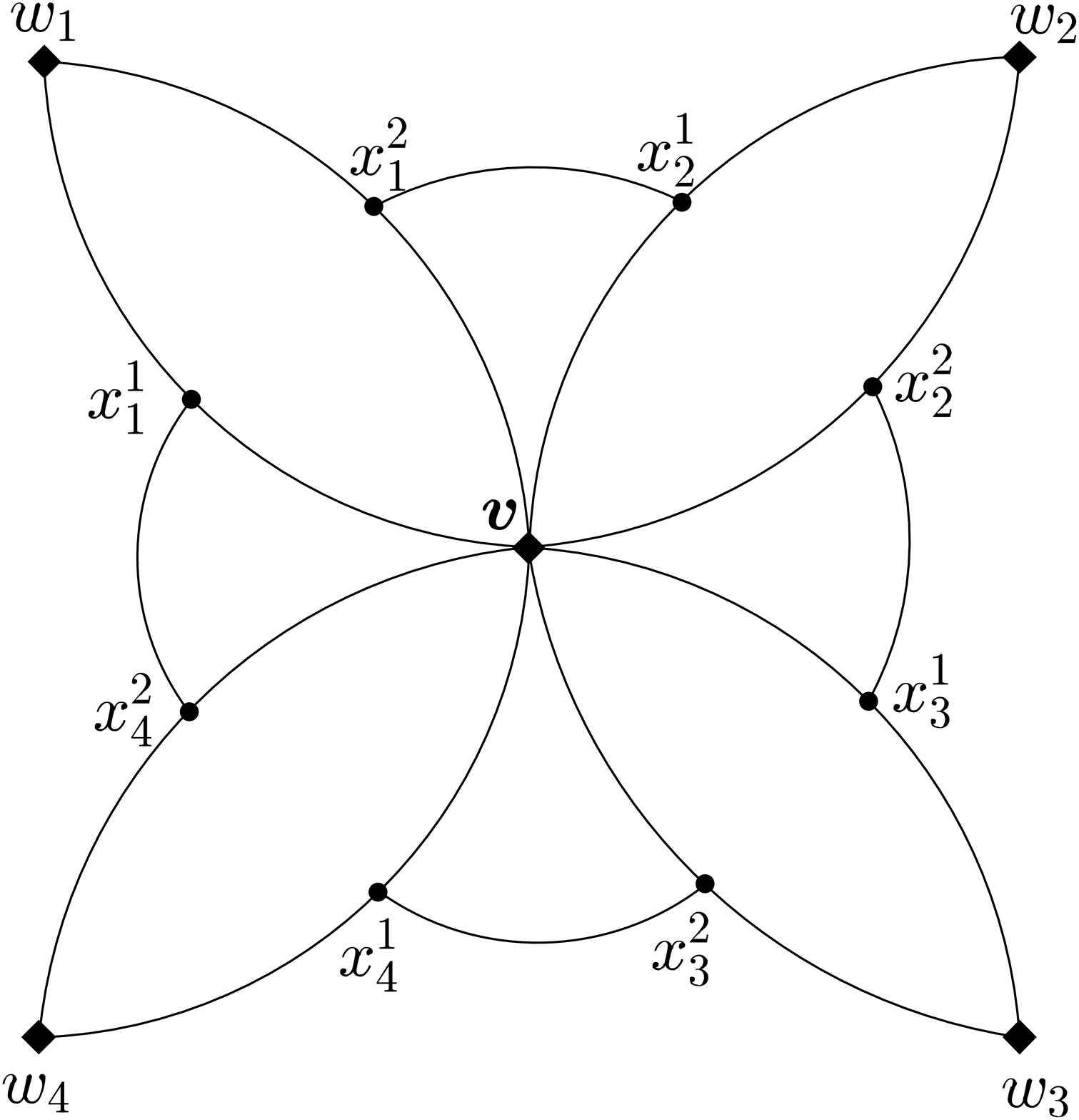}}
	\vspace*{8pt}
	\caption{The graph $ A_4 $}
	\label{fig:A4B4}
\end{figure}
\begin{figure}[h]
	\centerline{\includegraphics[width=0.7\textwidth]{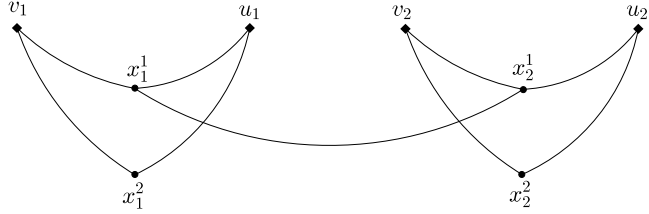}}
	\vspace*{8pt}
	\caption{The graph $ B$ }
	\label{fig:C1C2}
\end{figure}

\begin{definition}
	For a positive integer $k\geq 2$, let $A_k^W$ be the graph on the vertex set 
	$$V(A_k)=\{v, w_1,\dots ,w_k, x_1^1, \dots, x_k^1, x_1^2, \dots ,x_k^2\}$$
	and with the edge set 
	$$E(A_k)=\{vx_i^1, vx_i^2, w_ix_i^1, w_ix_i^2: 1\le i \le k\}\cup \{x_i^1x_{i+1}^2: 1\le i \le k\}\cup \{x_k^1x_1^2\}.$$
	The specified vertex set is $W_k = W =\{v\} \cup\{w_i: 1\le i \le k\}$ (for illustration see Fig.~\ref{fig:A4B4}).
\end{definition}
\begin{definition}
	Let $B^W$ be the graph  of order 8 with 
	$$V(B)=\{v_1, u_1, v_2, u_2, x_1^1,x_1^2,x_2^1,x_2^2\},$$
	$$E(B) = \{v_ix_i^1, v_ix_i^2,u_ix_i^1,u_ix_i^2: 1\le i \le 2\}\cup \{x_1^1x_2^1\}$$
	The specified vertex set is $W=\{v_1, u_1, v_2, u_2\}$ (for illustration see Fig.~\ref{fig:C1C2}).
	
\end{definition}
Note that $ A_k \in \cG_2(S_k)$ and $ B \in \cG_2(2K_2) $.

We first prove a lemma which will be referred to in the proof of our main theorem and also in later sections. 
\begin{lemma} \label{lem:d2dom}
	If $ G^D \in \mathcal{G}_1(F) $, then $D$ is a minimum $ 2 $-dominating set of $G$.
\end{lemma}
\begin{proof}
	By definition, every vertex from $X$ has two neighbors in $ D $. Thus, $D$ is a $ 2 $-dominating set in $G$.
	Suppose, to the contrary that, $ D' $ is a $ 2 $-dominating set of $ G $ such that $ |D'|<|D| $. Let $ D_1 = D \cap D'$ and $ D_2 = D \setminus D' $. Since $ D $ is independent in $ G $, the vertices in $ D_2 $ have to be $ 2 $-dominated by the vertices of $ D' \setminus D $, that is, every vertex in $ D_2 $ has at least two neighbors in $ D' $. Then we have \[|E[D',D_2]|\geq 2|D_2|.\]
	Moreover, by the definition of $\mathcal{G}_1(F) $, every vertex in $ D'\setminus D $ has exactly two neighbors in $ D $, so we have \[2|D'\setminus D|\geq |E[D',D_2]|.\]
	Thus, $ |D'\setminus D| \geq |D_2| $.
	Since $ D'=(D' \setminus D) \cup D_1 $, we conclude $ |D'| \geq |D_2|+|D_1|=|D| $, a contradiction.
\end{proof}
%

\begin{thm} \label{thm:2}
	Let $G^D$ be a graph from $\cH$. Then $\gamma(G)=\gamma_2(G)$ holds if and only if $G^D$ contains no subgraph isomorphic to $B^W$ and no subgraph isomorphic to $A_k^{W}$ for any $k\ge 2$.
\end{thm}
\begin{proof}
	Throughout the proof, we assume that $G\in \cH$ and hence there exists a $(C_3,C_4)$-free underlying graph $F$ such that $G \in \cG_2(F)$. By Lemma~\ref{lem:d2dom}, $D=V(F)$ is a minimum $ 2 $-dominating set of $G$.
	
	First assume that $G^D$ contains a (not necessarily induced) subgraph which is isomorphic to $B^W$. We may assume, without loss of generality, that this subgraph contains the vertices $ S = \{v_1, u_1, v_2, u_2, x_1^1,x_1^2,x_2^1,x_2^2\} $, the edges correspond  to those in Fig.~\ref{fig:C1C2}, and $ S \cap D =\{v_1, u_1, v_2, u_2\} $. Since $F$ is $ \{C_3,C_4\}$-free, the induced subgraph $F[S \cap D]$ is $ \{C_3,C_4\}$-free as well. Therefore, as $|S\cap D|=4$, $F[S \cap D]$ is a forest. It contains at least two edges, namely $v_1u_1$ and $v_2u_2$. Hence,  $F[S \cap D]$ contains a leaf, say $ v_1 $. Consider the set $ D' = (D \setminus S) \cup \{u_1, x_1^1, x_2^2\} $.  Observe that $ D' $ dominates all the vertices in $D$; the vertex $x_1^1 \in D'$ dominates $ x_2^1 $; the vertex $ u_1 $ dominates $ x_1^2 $. By the choice of $ v_1 $ and $u_1$, $ F[\{v_1,v_2,u_2\}] $ contains only the edge $ v_2u_2 $. Hence, all the subdivision vertices different from $ \{x_1^1,x_1^2,x_2^1,x_2^2\} $ are dominated either by $ D \setminus S $ or $ u_1 $. Hence, $D'$ is a dominating set in $G$ and $|D'|<|D|$. These imply $\gamma(G)<\gamma_2(G)$.
	
	Next assume that $G^D$ contains a subgraph which is isomorphic to $A_k^W$. We may assume, without loss of generality, that the vertices of this subgraph are named as given in the definition of $ A_k^W $. Consider the set $D'=(D\setminus W) \cup \{x_1^1, \dots x_k^1\}$.  Observe that $D'$ dominates all the vertices in $D$; the set $\{x_1^1, \dots x_k^1\} \subseteq D'$ dominates all the vertices of the form $x_i^s$ ($i \in [k]$, $s\in [2]$). Since $F$ is assumed to be $C_3$-free, for any further subdivision vertex $x_{i,j}^s$ of $G$, at least one of its neighbors which is a $D$-vertex, namely at least one of $v_i$ and $v_j$, is not included in $W$. Thus, $x_{i,j}^s$ is dominated by a vertex in $D\setminus W$. We may conclude that $D'$ is a dominating set in $G$. Since $|W|=k+1$, we have $|D'|<|D|$ from which $\gamma(G)<\gamma_2(G)$ follows. This finishes the proof of one direction of our theorem.
	
	For the converse, we assume that $G$ 	contains no subgraph isomorphic to $B^W$ and no subgraph isomorphic to $A_k^{W}$ for any $k\ge 2$, and then prove that $\gamma(G)=\gamma_2(G)$. In particular, having no subgraph isomorphic to $B^W$ means that every supplementary edge is inside a neighborhood of a $D$-vertex and, therefore, $N[x_{i,j}^s]\subseteq N[v_i]\cup N[v_j]$ holds for each supplementary vertex $x_{i,j}^s$. Now, suppose for a contradiction that $\gamma(G)<\gamma_2(G)$. Let $D'$ be a minimum dominating set of $G$ such that $|D'\cap D|$ is maximum under this condition. It is clear that  $|D'|=\gamma(G)<\gamma_2(G)=|D|$.
	
	We first prove that no pair $x_{i,j}^1$, $x_{i,j}^2$ are contained together in $D'$. Suppose, to the contrary, that $\{x_{i,j}^1$, $x_{i,j}^2 \}\subseteq D'$. Then, since $N[x_{i,j}^1]\cup N[x_{i,j}^2]\subseteq N[v_i]\cup N[v_j]$,
	the set $D''= (D'\setminus \{x_{i,j}^1, x_{i,j}^2\}) \cup \{v_i,v_j\}$ would be a dominating set of $ G $. This contradicts either the minimality of $ |D'| $ or the maximality of $|D'\cap D|$. 
	
	If we have some edges $v_iv_j \in E(F)$ such that $|X_{i,j}\cap D'|=0$, then we delete all these $X_{i,j}$ pairs  from $G$, delete all the associated edges from $F$ and obtain $G'$ and $F'$. Note that, by definition, $G'\in \cG_2(F')$ and $F'$ is still $(C_3,C_4)$-free. As $D'$ contains exactly one vertex from each remaining pair $X_{i,j}$, we infer that $|E(F')| \le |D'|$. By Lemma~\ref{lem:d2dom}, $\gamma_2(G')$ remains $|D|$ (we did not delete the possibly arising isolated vertices). We deleted only subdivision vertices not contained in  $D\cup D'$ and $D'$ contains exactly one vertex from each pair $X_{i,j}$ corresponding to an edge $v_iv_j \in E(F')$. Therefore,
	\begin{equation} \label{eq:1}
	 |E(F')| \le |D' \cap V(G')|< |D\cap V(G')|
	\end{equation}
	holds and $D' \cap V(G')$ is a dominating set in $G'$. By Lemma~\ref{lem:d2dom}, $D\cap V(G')$ remains a minimum $2$-dominating set in $G'$.

	$G'$ might contain several components. By the inequality (\ref{eq:1}), there is a component, say $G''$, such that $|D'\cap V(G'')| <|D \cap V(G'')|=\gamma_2(G'')$. It is clear that $G''$ is not an isolated vertex.
	Recall that $N_G[x_{i,j}^s] \subseteq N_G[v_i] \cup N_G[v_j]$ holds for each supplementary vertex $x_{i,j}^s$ in $G$ and hence, by construction, the analogous statement remains true in $G''$. Thus, the connectivity of the underlying graph $F''$ of $G''$ follows from the connectivity of $G''$. It also holds that $V(F'')= D \cap V(G'')$. Moreover, as $D' \cap V(G'')$ intersects each pair $X_{i,j} $ from $G''$, we have $|E(F'')| \le |D' \cap V(G'')|$. We may conclude 
\begin{equation} \label{eq:2}
|E(F'')| \le |D' \cap V(G'')| < |V(F'') |.
\end{equation}
The underlying graph $F''$ is therefore a tree and 
\begin{equation} \label{eq:3}
|E(F'')|= |D' \cap V(G'')| = |V(F'') |-1
\end{equation}
holds. By the first equality in (\ref{eq:3}), $D' \cap D \cap V(G'')=\emptyset$.  Note that $F''$ is not necessarily an induced subgraph of $F$ but, as $F$ is $C_3$-free, all the star-subgraphs of $F''$ are induced stars in $F$.

	Consider a non-pendant edge $ v_iv_j $ in $ F'' $ (if there exists). We know that $ D'\cap V(G'') $ is a dominating set in $G''$ and it contains exactly one vertex from $ X_{ij}$. Renaming the vertices if necessary, we may suppose $  x_{i,j}^1 \in D' $. Then the vertex $ x_{i,j}^2 $ must be dominated by a vertex from $ D' $, which is a neighbor of either $ v_i $ or $ v_j $. Without loss of generality, assume that $ x_{i,j}^2 $ is dominated by a neighbor of $ v_i $. Let $ S = V(G'')\setminus (N_{G''}(v_j) \setminus X_{ij}) $ and consider the induced subgraph $ G''[S] $. Let $ H $ be the component of the resulting graph, which contains both $ v_i $ and $ v_j $. 
	
Recall that $D' \cap V(G'')$ dominates all vertices in $G''$. By construction, $N_{G''}[v_p] \subseteq V(H)$ is true for every vertex $v_p\neq v_j$ from $ D' \cap V(H)$ and  
		$$N_{G''}[x_{p,q}^s] \subseteq N_{G''}[v_p] \cup N_{G''}[v_q] \subseteq V(H)$$
		 holds for every $x_{p,q}^s \in X \cap V(H)$ if $p \neq j \neq q$.  The set $D' \cap V(H)$ therefore dominates all vertices from $V(H) \setminus N_H[v_j]$. As $N_H[v_j]=\{v_j, x_{i,j}^1, x_{i,j}^2 \}$, it can be readily seen that $D'\cap V(H)$ is a dominating set in $H$.
		
		 Repeate sequentially this procedure of deleting non-pendant edges in the underlying graph. At the end we obtain a graph $ H_r $ with an underlying graph $ F_r $ such that $F_r$ is isomorphic to a star graph $ K_{1,m} $. Then the set $ D_r = V(H_r) \cap D' $ is a dominating set of $ H_r $ and it contains exactly one vertex from each pair $ X_{i,j} $ of subdivision vertices. 
	
	We will construct a directed graph $ R $ as follows. We create a vertex $ x_{i,j} $ corresponding to each pair $ X_{i,j} \subset V(H_r) $ of subdivision vertices. Then, we add a directed edge from $ x_{i,j} $ to $ x_{k,\ell} $ in  $R $, if the vertex in $ X_{i,j}\setminus D_r $ is dominated by the vertex in $ X_{k,\ell}\cap D_r $. As $ D_r $ has exactly one vertex from each pair $ X_{i,j} $, the outdegree of each vertex $ x_{i,j} \in V(R) $ is at least one. Thus, there is a directed cycle of order at least $ t \ge 2 $, which corresponds  to a subgraph isomorphic to $ A_t^W $ in $ H_r^{D \cap V(H_r)} \subseteq G^D $. This contradicts our assumption and finishes the proof of the theorem.
\end{proof}

\section{Algorithmic complexity}  \label{sec:3}

Since there are infinitely many forbidden subgraphs, Theorem \ref{thm:2} does not give directly a polynomial time recognition algorithm for $(\gamma,\gamma_2)$-graphs on $\cH$.  However, based on this characterization, we can design a polynomial time algorithm to check whether $ \gamma(G) = \gamma_2(G) $ holds for a general instance $ G^D \in \mathcal{H} $. 
\begin{thm}
	\label{thm:complexity}
	Let $ G^D \in \cH$ be given. It can be decided in polynomial time whether the graph $ G^D $ satisfies the equality $ \gamma(G) = \gamma_2(G) $.
\end{thm}
\begin{proof} 
	By Theorem \ref{thm:2},  $\gamma(G)=\gamma_2(G)$ holds if and only if $G^D$ contains no subgraph isomorphic to $B^W$ and no subgraph isomorphic to $A_k^{W}$ for any $k\ge 2$.
	
	\medskip
	\noindent	\hrulefill\\
	\noindent \textbf{Algorithm}
	
	\noindent	\hrulefill\\
	\indent \textit{Input:} A graph $ G^D \in \cH$ \\
	\indent \textit{Output:} If $ \gamma(G) = \gamma_2(G) $, then true; else false. \\
	\indent \hspace{1cm} for each supplementary edge $ uv $ in $ G $ \\
	\indent \hspace{2cm}if $D \cap (N_G(u) \cap N_G(v)) = \emptyset $, then return false\\
	\indent \hspace{1cm} for each vertex $ x $ in $ D $ \\
	\indent \hspace{2cm}  $ X \leftarrow N_G(x)$ and $ G'\leftarrow G[X] $\\
	\indent \hspace{2cm} $ k= (\deg_{G}x) /2$ \\
	\indent \hspace{2cm} for $ i \leftarrow 1 $ to $ k $ do \\
	\indent \hspace{3cm} $ E \leftarrow E(G') $\\
	\indent \hspace{3cm} for $ j \leftarrow 1 $ to $ k $ do \\
	\indent \hspace{4cm} if $ j\ne i $, then $ E \leftarrow E \cup \{x_j^1 x_j^2\}$   \\
	\indent \hspace{3cm} $ \mu $ $ \leftarrow  $ the order of the maximum matching  in $ E $\\
	\indent \hspace{3cm} if $ \mu = k $, then return false\\
	\indent \hspace{2cm}end-for\\
	\indent \hspace{1cm} end-for\\
	\indent \hspace{1cm} return true\\
	\indent end.
	
	\noindent	\hrulefill
	
	\medskip
	
	The  algorithm above, first, determines whether $B^W \subseteq G^D $. If it holds, then the algorithm halts. It can be readily checked that this part of the algorithm requires polynomial time.

	\noindent 	Then, in the next steps of the algorithm, the existence of subgraphs isomorphic to $A_\ell^{W}$ is tested. In order to find such a subgraph (if it  exists), the algorithm searches for an appropriate matching in $G[N_G(v_i)]$ for every vertex $ v_i$ from $ D $. Since a subgraph $A_\ell^W$ does not necessarily contain all the neighbors of  $ v_i $, it is not enough to check the existence of a perfect matching in $ G[N_G(v_i)] $. Instead, we define the edge set $E_i= \{x_{i,j}^1x_{i,j}^2: v_j \in N_F(v_i)\}$. Let $ G_i^* $ be the graph $ G[N_G(v_i)] $ extended by the edges from $ E_i $. Clearly, $G_i^* $ contains a perfect matching which is $ E_i $. On the other hand, $ G_i^* $ contains a perfect matching different from $ E_i $ if and only if $ G[N_G(v_i)] $ has a subgraph isomorphic to $ A_\ell^W $. Hence, the algorithm checks all possible $G_i^*-e$ graphs, where $e\in E_i$,  and if any of them has a perfect matching, then there exists a subgraph isomorphic to $ A_\ell^W $.

	In order to find a maximum matching in $G_i^*-e$, we can use Edmond's Blossom Algorithm \cite{Edmonds1965}, which was improved by Micali and Vazirani in \cite{Micali1980} to run in time $ O(\sqrt{n}m) $ for any graph of order $n $ and size $ m $. The procedure will be repeated $ (\deg_G(x) /2) = \deg_F(x)$ times for every vertex $ x\in D $, that is, $ \Sigma_{v\in V(F)}\deg(v) = 2|E(F)| $, in total. Thus, the second part of the algorithm requires polynomial-time. This finishes the proof.	
\end{proof}
\medskip
\noindent 	We now show that the same problem is NP-hard even on the graph class $ \mathcal{G}^D_1 $.

\begin{thm} \label{thm:4}
	Consider every graph $ G \in  \mathcal{G}_1 $ together with a specified set $ D $ such that $ G^2[D] \cong F $ and $ G \in \mathcal{G}_1(F) $. Then, it is NP-complete to decide whether the inequality $ \gamma(G) < \gamma_2(G)$ holds for a general instance $ G \in \mathcal{G}_1 $.
\end{thm}
\begin{proof}
	By Lemma \ref{lem:d2dom}, we have $ \gamma_2(G) = |D| $ and it can be checked in polynomial time whether a given set $ D' $ with $ |D'| < |D| $ is a dominating set of $ G $. Thus, the decision problem belongs to NP.

	In order to prove the NP-hardness, we present a polynomial-time reduction from the well-known $ 3 $-SAT problem, which is proved to be NP-complete \cite{Garey1979}.  
	
	Let $ X = \{x_1, x_2, \dots, x_k\}$ be a set of Boolean variables. A truth assignment for $ X $ is a function $ \varphi:X \rightarrow \{t,f\} $. If $ \varphi(x_i)=t $ holds, then the variable $ x_i $ is called $ true $; else if $ \varphi(x_i)=f$ holds, then $ x_i $ is called $ false $. If $ x_i $ is a variable in $ X $, then $ x_i $ and $ \neg{x_i} $ are literals over $ X $. The literal $ x_i $ is true under $ \varphi $ if and only if the variable $ x_i $ is true under $ \varphi $; the literal $ \neg x_i $ is true if and only if the variable $ x_i $ is false. A clause over $ X $ is a set of three literals over $ X $, represents the disjunction of those literals and it is satisfied by a truth assignment if and only if at least one of its members is true under that assignment. A collection $ \mathcal{C} $ of clauses over $ X $ is \textit{satisfiable} if and only if there exists some truth assignment for $ X $ that satisfies all the clauses in $ \mathcal{C} $. Such a truth assignment is called a \textit{satisfying truth assignment} for $ \mathcal{C} $. The $ 3 $-SAT problem is specified as follows.
	
	\medskip
	\medskip
	\noindent \textbf{3-SATISFIABILITY (3-SAT) PROBLEM}
	
\vspace{0.25cm}

\noindent\textbf{\textit{Instance:}} A collection $  \mathcal{C}  = \{C_1,C_2, \dots ,C_\ell\} $ of clauses over a finite set $ X $ of variables such that $ |\mathcal{C}_j| = 3 $, for $1 \le j \le \ell\} $.
\vspace{0.25cm}

\noindent\textbf{\textit{Question:}} Is there a truth assignment for $ X $ that satisfies all the clauses in $ \mathcal{C} $?
	
\vspace{0.5cm}
	Let $ \mathcal{C} $ be a $ 3 $-SAT instance with clauses $ C_1,C_2,\dots, C_\ell $ over the Boolean variables $ X=\{x_1,x_2, \dots, x_k\} $. We may assume that for every three variables $ x_{i_1}, x_{i_2}, x_{i_3} $ there exists a clause $ C_j $, where $ j\in [\ell] $, such that $ C_j $ does not contain any of the variables $ x_{i_1}, x_{i_2}, x_{i_3} $ (neither in positive form, nor in negative form). Otherwise, the problem could be reduced to at most eight (separated) $ 2$-SAT problems, which are solvable in polynomial time.
	
	We now construct a graph $ G \in \mathcal{G}_1(F) $, where $ F \cong S_{k+1} $, such that the given instance $ \mathcal{C} $ of $ 3 $-SAT problem is satisfiable if and only if $ \gamma(G) < \gamma_2(G)$. 
	The construction is as follows. 
	
	For every variable $ x_i $, we create three vertices $ \{x_i^t, x_i^f, v_i\} $ and then we add the edges $ x_i^tv_i $ and $ x_i^fv_i $. For every clause $ C_j \in \mathcal{C}$, we create a vertex $ c_j $, and if $ x_i $ is a literal in $ C_j $, then $ x_i^tc_j \in E(G) $; if $ \neg x_i $ is a literal in $ C_j $, then    $ x_i^fc_j \in E(G) $. Moreover, we add a vertex $ c^* $ and the edges $ c^*x_i^t  $ and $  c^*x_i^f $ for every $ i\in [k] $. We also add a vertex $ v_{k+1} $ and the edge set $  \{c_iv_{k+1}:1\le i\le\ell\} \cup \{c^*v_{k+1}\}$.  Finally, we add a new vertex $ v_0 $, which is adjacent to every vertex in $ V(G)\setminus \{v_1,v_2\dots,v_{k+1}\} $ (for an illustration of the construction see Fig.~\ref{fig:construction}). The order of $ G $ is obviously $ 3k+\ell +3 $ and this construction can be done in polynomial time. Note that $ G \in \mathcal{G}_1(F) $, where $ F $ is a star with center $ v_0 $ and leaves $ v_1,\dots, v_{k+1} $. Thus, we have $ \gamma_2(G)=k+2 $, by Lemma \ref{lem:d2dom}. 
	
	\begin{figure}[h]
		\centerline{\includegraphics[width=0.8\textwidth]{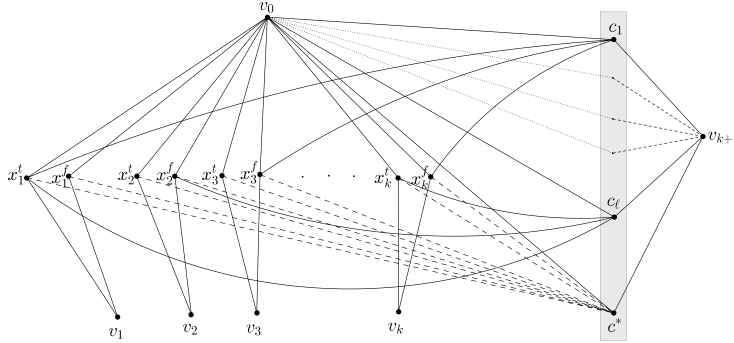}} 	
		\captionsetup{width=0.8\textwidth}
		\centering
		\caption{\protect An illustration of the construction for $ 3 $-SAT reduction: The clauses $ C_1 $ and $ C_\ell $ corresponding to the vertices $ c_1 $ and $ c_\ell $, resp., are $ C_1=(x_1 \vee \neg x_3 \vee \neg x_k) $ and  $ C_\ell=(x_1 \vee \neg x_2 \vee x_k) $.}
		\label{fig:construction}
	\end{figure}

	We now prove that $ \mathcal{C} $ is satisfiable if and only if $ \gamma(G) < \gamma_2(G) $. First, consider a truth assignment $ \varphi : x_i \rightarrow \{t,f\} $ which satisfies $ \mathcal{C} $. Let $D_1=\bigcup_{i\in [k] }\{x_i^t :\varphi(x_i)=t\} $ and  let $D_2 = \bigcup_{i\in [k] }\{x_i^f : \varphi(x_i)=f\}$. Consider the set $ D' = D_1 \cup D_2 \cup \{c^*\} $. It can be readily checked that $ D' $ is a dominating set of cardinality $ k+1 $. Hence, $ \gamma(G)< \gamma_2(G)$ follows.
	
	Conversely, assume that $ \gamma(G) < \gamma_2(G) $ and consider a minimum dominating set $ D' $ of cardinality at most $ k+1 $. In order to dominate $ v_i $, the set $ D' $ contains at least one vertex from the set $ \{x_i^t, x_i^f, v_i\} $, for each $ i\in [k] $. Similarly, to dominate $ v_{k+1} $, the set $ D' $ contains at least one vertex from the set $ \{c_1,c_2,\dots,c_\ell,c^*,v_{k+1}\}$. Since $ |D'|\le k+1 $, we have $|D' \cap \{x_i^t, x_i^f, v_i\}| = 1 $ for every $ i\in [k] $. Moreover, $ |D'\cap \{c_1,c_2,\dots,c_\ell,c^*,v_{k+1}\}| = 1$ and $ v_0 \notin D' $.
	
	Suppose that $ v_{k+1} \in D'$. In order to dominate the vertices $ x_i^t $  and  $ x_i^f $, the set $ D' $ contains the vertex $ v_i $ for all $ i \in [k] $. Hence, $ N_G(v_0) \cap D' = \emptyset $. From the discussion above, we know that $ v_0 \notin D' $. Thus, $ v_0 $ is not dominated by a vertex from $ D' $, a contradiction.
	
	Suppose that $ c_j \in D' $ for some $ j \in [\ell] $. Let $ C_j $ be the corresponding clause containing the variables $x_{i_1},x_{i_2}, x_{i_3} $. Consider any variable $ x_s \in X \setminus \{x_{i_1},x_{i_2}, x_{i_3}\} $. Since $|D' \cap \{x_i^t, x_i^f, v_i\}| = 1 $ for each $ i\in [k] $, $ D' $ contains $ v_s $ in order to dominate both of the vertices $ x_s^t $ and $ x_s^f $. By our assumption, there exists a clause $C_q  $ not containing the variables $x_{i_1},x_{i_2}, x_{i_3} $ neither in positive nor in negative form. Thus, $ c_q $ is not dominated by a vertex from $ D' $, a contradiction.
	
	Since $ |D'\cap \{c_1,c_2,\dots,c_\ell,c^*\}| = 1$, the only remaining case is $ c^*\in D'  $. Under this assumption, every vertex $ c_i $ must be dominated by the vertices corresponding to the literals in $ C_i$. Thus, the truth assignment
	\[ \varphi(x_i) =  \begin{cases} 
	t, & \text{if }x_i^t \in  D'  \\
	f, & \text{if }x_i^f \in D' \text { or  if }v_i \in  D' \\
	
	\end{cases}
	\]
	satisfies $ \mathcal{C} $. This finishes the proof.
\end{proof}

Theorem~\ref{thm:4} implies that it is coNP-complete to decide whether the equality $ \gamma(G)  = \gamma_2(G)$ holds for a general instance $ G $ from $  \mathcal{G}_1$. On the other hand, we cannot prove that the problem belongs to NP. Instead, we will consider the complexity class $\Theta^p_2$, which consists of those problems solvable by a polynomial-time deterministic algorithm using NP-oracle asked for only $O(\log n)$ times. (For a detailed introduction, please, see \cite{Marx2006}.) 

\begin{prop}
	The complexity of deciding whether $\gamma(G)=\gamma_2(G)$ holds for a general instance $G$ is in the class $\Theta^p_2$.			
\end{prop}	
\begin{proof}
	Using binary search, the parameters $\gamma(G)$ and $\gamma_2(G)$ can be determined by asking the NP-oracle $O(\log n)$ times whether the inequalities  $\gamma(G) \le k$ and $\gamma_2(G) \le k$ hold. Thus, the decision problem belongs to $\Theta^p_2$.
\end{proof}

Note that in \cite{Arumugam2013}, a similar statement was proved for the problem of deciding whether the transversal number $\tau(\cH)$ equals the domination number $\gamma(\cH)$ for a general instance hypergraph $\cH$.

\section{Characterization of $(\gamma, \gamma_2)$-perfect graphs} \label{sec:4}

Recently, Alvarado, Dantas, Rautenbach \cite{Alvarado2015-2, Alvarado2015} and Henning, J\"ager, Rautenbach \cite{Henning2018} studied graphs for which the equality between two fixed domination-type invariants hereditarily holds. The analogous problem for transversal and domination numbers of graphs and hypergraphs was considered in \cite{Arumugam2013}.

In this section, we characterize $(\gamma, \gamma_2)$-perfect graphs, that is, we characterize the graphs for which the equality between the domination and the $ 2 $-domination numbers hereditarily holds. By Lemma \ref{lem:0}, $ \delta(G) \geq 2$ is a necessary condition for $ \gamma(G)=\gamma_2(G) $. Hence, we define $(\gamma, \gamma_2)$-perfect graphs as follows. 

\begin{definition}
	Let $ G $ be a graph with $ \delta(G) \geq 2 $. Then $ G $ is a $(\gamma, \gamma_2)$-perfect graph if the equality $ \gamma(H)=\gamma_2(H) $ holds for every induced subgraph $ H $ of minimum degree at least two. 
\end{definition}
Note that a disconnected graph $ G $ is $(\gamma, \gamma_2)$-perfect if and only if all of its components are $(\gamma, \gamma_2)$-perfect.

In order to formulate the results of this section we will define the following class.
\begin{definition}
	Let $ S_{k} $ be the star  with center vertex $ v $ and end vertices $ \{v_1,v_2,\dots,\allowbreak v_k\} $ such that $ k\geq 1 $. Denote the edge $ vv_j\in  E(S_{k})$ by $ e_j $ for $ j\in [k] $. Let $ S(i_1,i_2,\dots,i_k) $ be the graph obtained by substituting each edge $ e_j $ of $ S_{k} $ by $ i_j $ parallel edges $ e_j^1,e_j^2,\dots e_j^{i_j} $, where $ i_j \geq 2 $, and then subdividing each edge $ e_j^r $ by adding the vertex $ x_j^r $ for all $ r \in [i_j] $ and all $ j\in [k] $. A graph $ G $ belongs to the class $ \cS $ if it is isomorphic to $ S(i_1,i_2,\dots,i_k) $ for some $ k\geq 1 $, where $ i_j\geq 2 $ for all $ j\in [k] $.
\end{definition}
We clearly have $  \cS \subseteq \cG_1$, since any $S(i_1,i_2,\dots,i_k)  \in \cG_1(F) $, where $ F \cong S_{k} $. On the other hand, if   $G' \in \cG(S_k)$, the underlying graph does not contain a clique of order larger than two and consequently, $|N(y)\cap D|=2$ for every supplementary vertex $y$. This implies that $G'\in \cG_1(S_k)$. By the definitions above, we have the following equivalence.
\begin{prop}
	\label{prop:perfect}
	For any graph, $G\in \cS$ holds if and only if $G\in \cG_1(S_k)$ $ ( $or, equivalently, $G\in \cG(S_k))$ for a non-trivial star $S_k$ and $G$ does not contain a supplementary edge.
\end{prop}
%
The main result of this section is a characterization theorem for $(\gamma,\gamma_2)$-perfect graphs.

\begin{thm} 
	\label{thm:perfect}
	$ G $ is a connected $ (\gamma,\gamma_2)$-perfect graph if and only if $ G \in \cS. $
\end{thm}
\begin{proof}
	We first prove that if $ G \in \cS $, then it is $ (\gamma,\gamma_2) $-perfect graph. 
	By Proposition \ref{prop:perfect}, we know that $ G \in \cG_1(F) $, where $ F \cong S_{k} $ for $ k\geq 1 $. Then, by Lemma \ref{lem:d2dom},   $ \gamma_2(G) = |V(F)|= k+1$. Since a minimal $ 2 $-dominating set is a dominating set, we have the inequality $ \gamma(G) \leq k+1 $. In order to prove that $\gamma(G)=\gamma_2(G)$, it is enough to show that $ \gamma(G) > k $. Suppose, to the contrary, that $ D' $ is a minimum dominating set of $ G $ such that $ |D'|\leq k $.
	
	Consider the vertices of $ G $ corresponding to the end vertices of the star $ S_{k} $. Let $ \{v_1,v_2,\dots, v_k\} = V(F)\setminus \{v\} \subseteq V(G)$, where $ v $ is the center of $ F\cong S_{k} $. Since $D'$ is a dominating set,  $ |N_G[v_j]\cap D'| \geq 1 $ for each $ j \in [k] $. Note that the closed neighborhoods of any two vertices from the set $ \{v_1,v_2,\dots,v_k\} $ are disjoint. Since $ |D'|\leq k $ by our assumption, we have $v\notin D'$ and $ |N_G[v_j] \cap D'|=1 $, for every $ j\in[k] $. Moreover, as the center $v$ must also be dominated, there exists some $ j \in [k] $ and $ r\in [i_j] $ such that  $ x_j^r \in D' $. Then, $ v_j \notin D' $  and the vertices in $ (X_j \cup Y_j)\setminus \{x_j^r\} $ are not dominated by $ D' $, which is a contradiction. Consequently, $ k $ vertices are not enough to dominate all the vertices of $ G $, that is, $ \gamma(G) \geq k+1 $. It follows that $ \gamma(G) = \gamma_2(G) $ for any $ G \in \cS $.
	
	Next, suppose that $ H $ is an induced subgraph of $ G $ with minimum degree at least two. If $H$ does not contain any subdivision vertices, we have $ \delta(H)=0 $,  a contradiction. Thus, $ H $ contains a subdivision vertex. Let $ x_p^q \in V(H)$ for some $ p\in [k] $ and $ q\in [i_p] $. Since $ \deg_G(x_p^q)=2 $, then both of the neighbors of $ x_p^q $ must be in $ V(H) $, i.e., $ N_G(x_p^q) = \{v,v_p\}\subseteq V(H) $. Since $ \delta(H) \geq 2 $ by the assumption, using an argument similar to the above, we have $ \deg_H(v_p)\geq 2 $. Thus, $ |(X_p \cup Y_p)\setminus \{x_p^q\})\cap V(H)| \geq 1 $. Consequently, $ H \in \cS $ and, as it was proved above, $ \gamma(H) = \gamma_2(H) $ holds for every induced subgraph of $ G $ with minimum degree at least two.
	
	To prove the converse, assume that $ G $ is a connected $ (\gamma, \gamma_2) $-perfect graph.	Note that $ \gamma(C_n) = \ceil{\frac{n}{3}} \ $ and $ \gamma_2(C_n) = \ceil{\frac{n}{2}}$, where $ n\ge 3 $. Thus, the $(\gamma,\gamma_2)$-perfect graph $ G $ does not contain an induced cycle $ C_n $, where $ n =3 $ or  $ n\geq 5 $.
	
	\medskip
	\begin{figure}[h]
		\centerline{\includegraphics[width=0.5\textwidth]{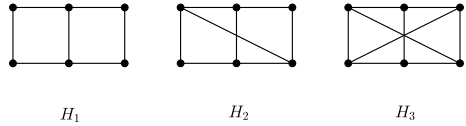}}
		\caption{The graphs $ H_1 $, $ H_2 $ and $ H_3 $}
		\label{fig:induced3Graphs}
	\end{figure}
	Now, suppose that $ G $ has a non-induced subgraph isomorphic to $ C_r $, for some $ r\geq 5 $. Since all of its induced cycles are $ 4 $-cycles, $ G $ contains at least one of the three graphs $ H_1, H_2 $ and $ H_3 $, shown in Figure \ref{fig:induced3Graphs}, as an induced subgraph. Observe that $ \gamma(H_i) < \gamma_2(H_i) $ for all $ i \in \{1,2,3\} $. This contradicts our assumption that $ G $ is a $ (\gamma, \gamma_2)$-perfect graph. Thus, $ G $ does not contain a cycle $ C_r $, where $ r\neq 4 $.
	
	Since $ G $ is $ (\gamma, \gamma_2)$-perfect by the assumption, then the equality $ \gamma(G)=\gamma_2(G) $ holds. By Proposition \ref{prop:1}, we know that $ G\in \cG $. Thus, if $ D $ is a minimum $ 2 $-dominating set of $ G $, then $ D $ is independent and $ F=G^2[D] $ is the underlying graph of $ G $.
	
	First, note  that $ F $ does not contain a cycle $ C_r $ for $ r\geq 3 $. Otherwise, $ G $ would contain a subgraph isomorphic to $ C_{2r} $, which is a contradiction. Thus, $ F $ is a forest and $G\in \cG_1(F)$. Then suppose that $ F $ is not connected. Since $ G $ is connected, there is a supplementary edge $ e = uv $, where $ u $ and $ v $ are two subdivision vertices of $ G $ such that $N(u)\cap V(F)$ and $N(v)\cap V(F)$ are in different components of $ F $. By the definition of the graph class $ \cG_1$, there are two vertices $ u' $ and $ v' $ such that $N_G(u) \cap V(F) = N_G(u') \cap V(F)  $ and $N_G(v) \cap V(F) = N_G(v') \cap V(F)  $. Let $ \{x_1,x_2\} = N_G(u) \cap V(F) $ and $ \{x_3,x_4\} = N_G(v) \cap V(F) $, where the sets $ \{x_1,x_2\} $ and $ \{x_3,x_4\} $ are contained by different components of $ F $. Consider the set $ A = \{x_1,x_2,x_3,x_4,u,v,u',v'\} $ and the induced subgraph  $ G[A] $. It is easy to check that $\delta(G[A]) \ge 2$,  $\gamma(G[A])\le 3$ and $\gamma_2(G[A])=4 $, which is a contradiction. Thus, $ F $ is a tree.
	
	Suppose that $ G $ has a supplementary edge $ e=uv \in E(G) $, where $ u,v \in V(G)\setminus V(F)  $. Let $ N_G(u) \cap V(F) = \{x_1,x_2\} $ and $ N_G(v) \cap V(F) = \{x_3,x_4\} $. Note that $|\{ x_1,x_2\} \cap \{x_3,x_4\}|\leq 1 $, otherwise $ G $ would contain a subgraph isomorphic to $ C_3 $. By Lemma~\ref{lem:2nghbrs}, there exist two further vertices $u'$ and $v'$ satisfying $ N_G(u') \cap V(F) = \{x_1,x_2\} $ and $ N_G(v') \cap V(F) = \{x_3,x_4\} $.  If  $|\{ x_1,x_2\} \cap \{x_3,x_4\}|= 1 $, then without loss of generality, assume that $ x_2 = x_3 $. Then, there is a subgraph of $ G $ isomorphic to $ C_3 $ induced by the vertices $ u$, $v$ and  $x_2 $, a contradiction. If  $\{ x_1,x_2\} \cap \{x_3,x_4\} = \emptyset$, then let $ S= \{x_1,x_2,x_3,x_4,u,v,u',v'\}$. A similar argument applied to the  subgraph  of $ G $ induced by the vertex set $ S $ yields the inequality $ \gamma(G[S])\le 3 < \gamma_2(G[S])=4$. Thus, $ G $ does not have any supplementary edges.
	
	\medskip
	\begin{figure}[h]
		\centerline{\includegraphics[width=0.3\textwidth]{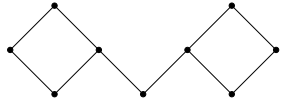}}
		\caption{The graph $ H_4 $}
		\label{fig:H4}
	\end{figure}
	
	Suppose that $ F $ contains a subgraph isomorphic to $ P_4 $. Since $ G $ does not have a supplementary edge, it contains an induced subgraph isomorphic to $ H_4 $ given in Figure \ref{fig:H4}. Note that $ \delta(H_4) \geq 2 $  and $ 3=\gamma(H_4) < \gamma_2(H_4) = 4 $, which contradicts the assumption that $ G $ is $ (\gamma, \gamma_2)$-perfect. Thus, $ F $ is a star, $G\in \cG_1(F)$, and $G$ does not contain supplementary edges. This finishes the proof by Proposition~\ref{prop:perfect}.
\end{proof}

The graph obtained from an edge by attaching two pendant edges to both of its ends will be called $ T_6 $ (for illustration see Fig.~\ref{fig:T}).

\medskip
\begin{figure}[h]
	\centerline{\includegraphics[width=0.15\textwidth]{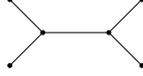}}
	\caption{The graph $ T_6 $}
	\label{fig:T}
\end{figure}
\begin{prop}
	\label{prop:S}
	$ G \in \cS $ if and only if $ G $ is a connected graph with $ \delta(G)\geq 2 $ and it contains no subgraph isomorphic to any of  $T_6,P_8$, or  $C_k $ where $ k\neq 4 $.
\end{prop}
\begin{proof}
	If $ G \in \cS $, then it is easy to see that $ G $ is a connected graph with $ \delta(G)\geq 2 $ and it does not contain a subgraph isomorphic to  $T_6,P_8$, or  $C_k $ where $ k\neq 4 $.
	
	Now, assume that $ G $ is a connected graph of minimum degree at least two which does not contain a subgraph isomorphic to $T_6,P_8$, or  $C_k $ where $ k\neq 4 $. Note that $G$ is bipartite. We further have $\min\{\deg_G(u),\deg_G(v)\}=2$ for each edge $uv \in E(G)$, since $\delta(G) \ge 2$ and $G$ does not contain a subgraph isomorphic to $T_6$ or $C_3$.
	
	First, suppose that $ G $ contains an edge $ e=uv \in E(G) $, which is a bridge. Then $ G-e $ has two components, say $ G_1 $ and $ G_2$. Since $ \delta(G)\geq 2 $, both $G_1$ and $G_2$ are non-trivial graphs and may contain at most one vertex, namely either $u$ or $v$, which is of degree 1. Thus, both of the components contain a cycle. These cycles must be vertex-disjoint $ 4 $-cycles with a path between them. Hence, $ G $ contains a subgraph isomorphic to $ P_8 $ and this contradicts our assumption.

	Since $ G $ does not contain a bridge, every edge of $ G $ lies on a $ 4 $-cycle. If all the vertices of $ G $ have degree two, then $ G $ is isomorphic to $ C_4 $ and $ G \in \cS $. If $ G $ is not isomorphic to $ C_4 $, then every $ 4 $-cycle contains a vertex of degree at least three. For a vertex $ v $ of degree two, we define the function $ f(v) $ to denote the vertex opposite to $ v $ in a $4 $-cycle. Let $ A = \{v \in V(G): \deg(v) \geq 3$  or  $\deg(f(v))\geq 3\} $. 
	Consider two vertices $ u,v \in A $. If $ uv \in E(G) $, then $ uv $ belongs to a $ 4 $-cycle, say $ uvv'u' $. At least one of $ u $ and $ v $ is of degree two, without loss of generality, say $ \deg(u)=2 $. Thus, $ u $ belongs only to this $ 4 $-cycle. Since $ f(u) = v' $, by the definition of $ A $, $ \deg(v')\geq 3 $. If $ \deg(v)\geq 3 $, then $ vv' \in E(G) $, we have a contradiction. If $ \deg(v)= 2 $, then  $ v\in A $ and $ v $ belongs only to the $ 4 $-cycle $ uvv'u' $. Thus, $ f(v)=u' $,  $ \deg(u')\geq 3 $ and $ u'v' \in E(G) $, which is a contradiction. Hence, $ A $ is independent.
	Consider two vertices $ u,v \in V(G) \setminus A $. If $ uv \in E(G) $, then at least one of $ f(u) $ or $ f(v) $ is of degree at least three. Then, by the definition of the function $ f $, we have $ u \in A $ or $ v \in A $, which is a contradiction. Hence, $ V(G) \setminus A $ is independent.
	
	Consequently,  $ (A, V(G)\setminus A) $ is a bipartition of $ V(G) $.  Note that every $ 4 $-cycle has exactly two vertices in $ A $. Hence, $G^A \in \cG_1(F)$ where $F\cong G^2[A]$, and there are no supplementary edges. Since $G$ does not have a subgraph isomorphic to $C_n$ for $n\ge 6$, the underlying graph is a tree. If $ F $ contains a subgraph isomorphic to $ P_4 $, then $ G $ contains a subgraph isomorphic to $ P_8 $, which is a contradiction. Thus, $ F $ is a star, and Proposition~\ref{prop:perfect} implies that $G\in \cS$.
\end{proof}
Thus, Proposition~\ref{prop:S} allows us to state Theorem~\ref{thm:perfect} in a different form as follows.	
\begin{thm}
	Let $ G $ be a connected graph with $ \delta(G)\geq 2 $. Then $ G $ is a $ (\gamma,\gamma_2 )$-perfect graph if and only if $ G $ contains no subgraph isomorphic to any of  $T_6,P_8$, or  $C_k $ where $ k\neq 4 $.
\end{thm}		

Note that for any $G\in \cS$, the center of the underlying star can be chosen as a vertex $v$ of degree $\Delta(G)$ and then, the subdivision vertices are exactly those contained in $N_G(v)$. Therefore, the characterization given in Theorem~\ref{thm:perfect} directly yields a polynomial-time algorithm which recognizes $(\gamma, \gamma_2)$-perfect graphs.

\section{Concluding remarks and open problems} \label{sec:5}

In Section 1, we defined the graph class $\cG$ which contains all $(\gamma, \gamma_2)$-graphs. Then, in Section 2, we gave a characterization for $(\gamma, \gamma_2)$-graphs over a specified subclass $\cH$ of $\cG$. In the definition of $\cH$ and in the proof of the main theorem, we referred to the properties of the
\begin{figure}
	\centerline{\includegraphics[width=0.7\textwidth]{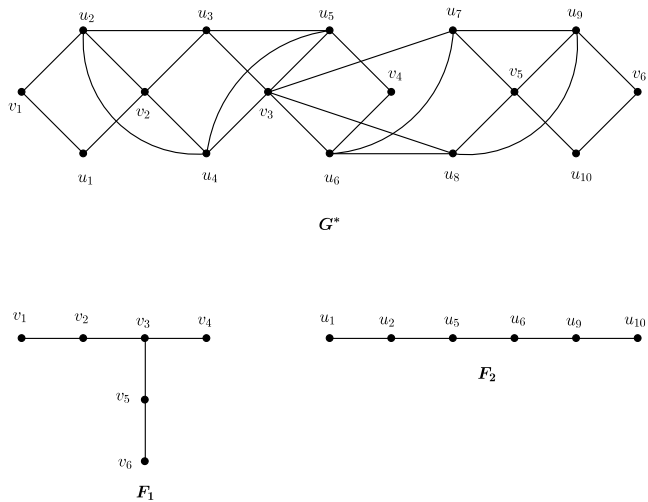}}
	\caption{$ G^* $ is a graph with $ \gamma (G^*)= \gamma_2 (G^*) = 6$, which has two non-isomorphic underlying graphs and $ G^* \in \cH(F_1) \cap \cH(F_2) $.}
	\label{fig:nonIso}
\end{figure}
underlying graph. We noted there that the underlying graph is not always unique when a graph $G$ from $\cG$ is given.  In Figure \ref{fig:nonIso}, we show a $ (\gamma,\gamma_2) $-graph having two non-isomorphic underlying graphs. Analogously, one can construct infinitely many graphs with the same property.

In the definition of the class $\cH$, we forbid $ 3 $-cycles and $ 4 $-cycles in the underlying graph. The characterization given in Theorem~\ref{thm:2} does not hold if $ 3 $-cycles are not forbidden in the underlying graph. This is shown by the graph $A_4^* \in \cG_2(F)$ (see Figure~\ref{fig:A4_star}), where the underlying graph $F$ is a star supplemented by an edge. One can readily check that even if $A_4^*$ contains an induced $A_4^W$ subgraph, it remains a $(\gamma, \gamma_2)$-graph as $\gamma(A_4^*)=\gamma_2(A_4^*)=5$. Similarly, it is possible to construct graphs whose underlying graphs are $C_3$-free but not $C_4$-free such that the statement of Theorem~\ref{thm:2} does not remain valid for them. Therefore, the following problems are still open.

\begin{figure}[h]
	\centerline{\includegraphics[width=0.3\textwidth]{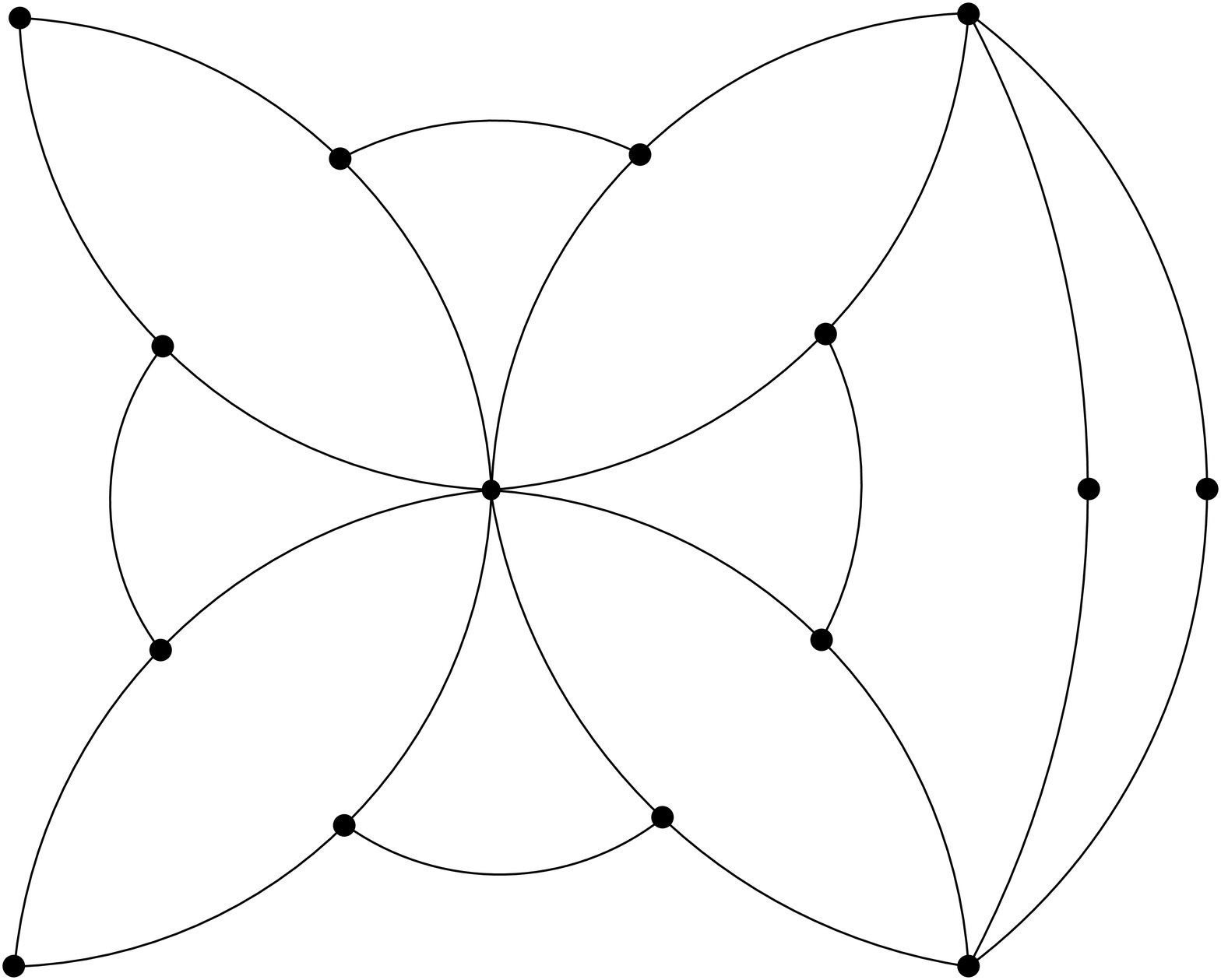}}
	\caption{The graph $ A_4^* $ }
	\label{fig:A4_star}
\end{figure}  
\begin{prob}
	Characterize $(\gamma, \gamma_2)$-graphs over the following graph classes:
	\begin{enumerate}
		\item Over the subclass of $\cG_2$ where the underlying graph does not contain any $C_4$ subgraphs;
		\item Over the subclass of $\cG_2$ where the underlying graph is $C_3$-free;
		\item Over $\cG_2$.
	\end{enumerate}
\end{prob}

\vspace{1.5cc}

\noindent \textbf{Acknowledgment.} Research of Csilla Bujt\'as was partially supported by the Slovenian Research Agency under the project N1-0108.

\vspace{2cc}

\bibliographystyle{abbrv}
\bibliography{references}

\end{document}